\documentclass[10pt]{amsart}
\usepackage{amsthm,mathrsfs,graphicx,mathabx,tikz,enumitem,tikz,hyperref,cancel,graphics,bm}

\usepackage[normalem]{ulem}
\usepackage{amsmath,bbm}
\usepackage{url}

\usepackage{setspace}
\newcommand{\excise}[1]{}

\usepackage[margin=1.25in]{geometry}

\newtheorem{theorem}{Theorem}
\numberwithin{theorem}{section}

\newtheorem{corollary}[theorem]{Corollary}
\newtheorem{lemma}[theorem]{Lemma}

\theoremstyle{definition}
\newtheorem{definition}[theorem]{Definition}

\newtheorem{example}[theorem]{Example}

\DeclareMathOperator{\indeg}{indeg}

\title{Reconstructing nearly simple polytopes from their graph}
\author{Joseph Doolittle}
\begin{document}
\begin{abstract}

   We present a partial description of which polytopes are reconstructible from their graphs. This is an extension of work by Blind and Mani (1987) and Kalai (1988), which showed that simple polytopes are reconstructible from their graphs.
   
   In particular, we introduce a notion of $h$-nearly simple and prove that 1-nearly simple and 2-nearly simple polytopes are reconstructible from their graphs. We also give an example of a 3-nearly simple polytope which is not reconstructible from its graph. Furthermore, we give a partial list of polytopes which are reconstructible from their graphs in a non-constructive way.
\end{abstract}
\maketitle 

\section{Introduction}
In 1987, Blind and Mani published a proof that the face lattice of simple polytopes can be determined from the graph of the polytope \cite{BM}. In 1988, Kalai published a simple proof of the same fact \cite{Kalai}. In that paper, Kalai noted that the k-faces of $P$ were given by k-regular subgraphs of its graph, $G_P$, with the condition that they were initial with respect to a ``good" orientation. The definition given for ``good" is equivalent to being acyclic and minimizing the value of a function of orientations which Kalai called $f^O$. Work on this subject has broadly ignored reconstruction of non-simple polytopes from their graphs, sometimes going so far as to assume that only simple polytopes are reconstructible from their graphs  \cite{Friedman}.

In this paper, we show that there are two more classes of polytopes which are reconstructible from their graphs. A $d$-dimensional convex polytope is $h$-nearly simple if all but $h$ of its vertices are contained in exactly $d$ edges. The main results of this paper are that $1$-nearly simple polytopes are reconstructible from their graphs, that $2$-nearly simple polytopes are reconstructible from their graphs, and that there are $3$-nearly simple polytopes that are not reconstructible from their graphs. See Theorem \ref{1simple}, Theorem \ref{2simple}, and Example \ref{example}, respectively.

Furthermore, the ``goodness" that Kalai noted about $f^O$-minimizing orientations is expanded upon in Lemma \ref{fOimplies}. These orientations are of particular interest, and Joswig, et al. provide a characterizations of these orientations based on the sinks of 2-faces of a simple polytope \cite{Joswig}. Determining a way to obtain and verify these orientations in polynomial time is an ultimate goal. A conjecture of Perles would give such a polynomial time method, but a counter-example to the conjecture was given by Haase and Ziegler \cite{HZ}.

The methods used to obtain the first two results focus on determining orientations which minimize $f^O$. Finding such orientations allows us to determine some 2-faces of $P$, which in turn allows us to determine the graph of the polytope arising from the truncation of $P$ at a non-simple face. The last result comes from an analysis of data provided by Miyata et al. \cite{Miyata}.

\section{Definitions}
\begin{definition}
The edge-vertex graph (or graph) of a polytope $P$ is the graph $G_P=(V_P,E_P)$ where $V_P=\{$vertices of $P\}$ and $E_P=\{{v_i,v_j}|$there is a 1-dimensional face of $P$ that contains $v_i$ and $v_j\}$.
\end{definition}

\begin{definition}

Let $P$ be a convex polytope and let $v$ be an element of $\mathbb{R}^n$ whose dot product with each vertex of $P$ is distinct.  The orientation of the graph $G_P$ given by orienting each edge to point towards the vertex of the edge whose dot product with $v$ is larger is a \textit{linear functional orientation}.
\end{definition}

For this to be a good definition, there must exist such a vector for any polytope. The set of vectors whose dot product with the vertices of $P$ is not an injective function is the union of a finite number of $(d-1)$-dimensional vector spaces. Therefore there exists some vector whose dot product with the vertices of $P$ is injective, and therefore each edge of $G_P$ has a well defined orientation with respect to that vector. So $G_P$ has a well defined linear functional orientation.

\begin{definition}
Let $O$ be an orientation of a graph $G=(V,E)$. Define $f^O:= \sum_{v \in V} 2^{\indeg(v)}$.
\end{definition}

\begin{definition}
A $d$-dimensional polytope $P$ is \textit{simple} if each vertex of $P$ has degree $d$ in $G_P$.
\end{definition}

\begin{definition}
A $d$-dimensional polytope $P$ is \textit{simple at vertex $v$} if $v$ has degree $d$ in $G_P$.
\end{definition}

If $P$ is simple at vertex $v$, then each subset of the edges of $P$ containing $v$ is contained in a unique face of $P$ that has dimension equal to the size of the subset.

\begin{definition}
A \textit{sink of a graph $G$} with respect to an orientation $O$ is a vertex of $G$ for which all edges incident to it in $G$ are oriented towards it with respect to $O$.
\end{definition}

\begin{definition}
A \textit{sink of a polytope $P$} with respect to an orientation $O$ is a sink of the graph of $P$ with respect to $O$.
\end{definition}

\begin{definition}
Let $O$ be an orientation of a graph $G$. A set $J \subset V_G$ is an \textit{initial set} of $G$ with respect to $O$ if each edge of $G$ with one endpoint in $J$ and one endpoint in $G \setminus J$ is oriented away from $J$ with respect to $O$.
\end{definition}

The following theorem is due to Blind and Mani '87 \cite{BM} and Kalai '88 \cite{Kalai}.

\begin{theorem}
\label{Kalai}
Let $P$ be a simple d-dimensional polytope. Then the face lattice of $P$ can be determined from the graph of $P$. 
\end{theorem}

\begin{lemma}
\label{fOfaces}
Let $P$ be a $d$-dimensional polytope, and $O$ be an acyclic orientation of $G_P$. Then $f^O \geq \sum_{v \in V_P} |\{$faces $F$ of $P | v$ is a sink of $G_F$ with respect to $O\}| \geq \sum_{i=0}^d f_i(P)$, where $f_i$ is the number of $i$-faces of $P$. 
\end{lemma}

\begin{proof}

There is an injective function between faces of $P$ with $v$ as a sink of their graph and subsets of the edges of $P$ oriented towards $v$, given by selecting the edges oriented towards $v$ that are in the face. Since $2^{\indeg(v)}$ is the size of the codomain of this function, the number of faces of $P$ with $v$ as a sink is less than or equal to $2^{\indeg(v)}$. Summing the number faces with $v$ as a sink of their graph and $2^{\indeg(v)}$ over all vertices of $P$ gives $f^O \geq \sum_{v \in V_P} |\{$faces $F$ of $P | v$ is a sink of $F$ with respect to $O\}|$. Each face of $P$ has at least one sink since $O$ is acyclic. Then $\sum_{v \in V_P} |\{$faces $F$ of $P | v$ is a sink of $F$ with respect to $O\}| \geq \sum_{i=0}^d f_i$ since each face of $P$ appears at least once in the left hand sum and appears exactly once in the right hand sum.
\end{proof}

\begin{lemma}
\label{fOimplies}
Let $P$ be a $d$-polytope and $O$ an acyclic orientation of its graph. Then $f^O = \sum_{i=0}^d f_i$ if and only if the following conditions hold:

\begin{itemize}
\item Each face of $P$ has a unique sink.
\item Each face of $P$ is simple at its sink.
\item For each vertex $v$ of $P$, there is a face of $P$ with $v$ as its sink and containing each edge of $P$ oriented towards $v$.
\end{itemize}
\end{lemma}

\begin{proof}
$\Rightarrow$ Assume that $f^O = \sum_{i=0}^d f_i$. Then by Lemma \ref{fOfaces}, $\sum_{i=0}^d f_i = \sum_{v \in V_P} |\{$faces $F$ of $P | v$ is a sink of $F$ with respect to $O\}|$. Therefore every face of $P$ has a unique sink. 

Since $f^O = \sum_{i=0}^d f_i$, by Lemma \ref{fOfaces}, $\sum_{v \in V_P} 2^{\indeg(v)} = \sum_{v \in V_P} |\{$faces $F$ of $P | v$ is a sink of $F$ with respect to $O\}|$. Since there are at most $2^{\indeg(v)}$ faces of $P$ that have $v$ as its sink, the summands must be equal. That is, for all $v$, $2^{\indeg(v)} = |\{$faces $F$ of $P | v$ is a sink of $F$ with respect to $O\}|$. Each face of $P$ that has $v$ as a sink uniquely specifies a subset of edges oriented towards $v$  by containment in that face. Since there are $2^{\indeg(v)}$ distinct subsets of edges oriented towards $v$, and there are $2^{\indeg(v)}$ faces of $P$ with $v$ as a sink, there is a face $F_v$ of $P$ that has $v$ as its sink and contains all the edges oriented towards $v$ in $G_P$. Every subface of $F_v$ which contains $v$ has $v$ as a sink and maps to a subset of edges oriented towards $v$. Since this is injective, and the sets are the same size, this is a bijective map. So the rank of $F_v$ in the face poset of $P$ is indeg$(v)+1$ which is one more than the degree of $v$ in $G_{F_v}$; therefore the dimension of $F_v$ is the degree of $v$ in $G_{F_v}$, so $F_v$ is simple at $v$. Each face that has $v$ as its sink is a subface of $F_v$, and so is simple at $v$. So each face of $P$ is simple at its sink.

$\Leftarrow$ We write $f^O$ as $\sum_{v \in V_P} 2^{\indeg(v)}$ and expand the summand to be the sum of 1 over all choices of subsets of edges oriented towards $v$. For each $v$, let $F_v$ be the face of $P$ with $v$ as its sink containing all edges oriented towards $v$. Then $F_v$ is simple at $v$ by assumption, and any subset of edges oriented towards $v$ determines a subface of that face. So $\sum_{v \in V_P} 2^{\indeg(v)} \leq \sum_{v \in V_P} |\{$faces $F$ of $P | v$ is a sink of $F$ with respect to $O\}|$. Since each face has a unique sink, $ \sum_{v \in V_P} |\{$faces $F$ of $P | v$ is a sink of $F$ with respect to $O\}| = \sum_{n=0}^d f_i$. So $f^O \leq \sum_{i=0}^d f_i$. By Lemma \ref{fOfaces}, this implies that $f^O = \sum_{i=0}^d f_i$.
\end{proof}

\begin{lemma}
\label{affine}
Let $K$ be a finite set of $n$-dimensional affine spaces contained in an $(n+m+1)$-dimensional vector space $V$.  Let $T$ be an open subset of $V$. Then there exists an $m$-dimensional affine space whose intersection with the union of $K$ is empty and whose intersection with $T$ is nonempty. Furthermore, there is such a space parallel to any $m$-dimensional affine space.
\end{lemma}

\begin{proof}
Let $H$ be an arbitrary $m$-dimensional affine space in $V$. Let $\pi$ be a projection of $V$ so that $\pi(H) = \{h\}$, a single point in $V$, and so that $\pi(V)$ is an $(n+1)$-dimensional vector space. Since $T$ is open, $\pi(T)$ is open in $\pi(V)$. Since $H$ could at worst be normal to affine spaces in $K$, $\pi(\bigcup K)$ is the union of a set of affine spaces of dimensions possibly ranging from 0 to $n$. Since $\pi(V)$ is an $(n+1)$-dimensional space, and $\pi(\bigcup K)$ is the union of affine spaces of dimension at most $n$, there is a point in $\pi(V)$ contained in $\pi(T)$ and not contained in $\pi(\bigcup K)$; call it $h'$. Take $H'$ to be the pre-image of $h'$ under $\pi$; $H' = \pi^{-1}(h')$. Then $H'$ nontrivially intersects $T$, and $H'$ intersected with $\bigcup K$ is empty. So $H'$ is the desired $m$-dimensional affine space. By construction, $H$ and $H'$ are parallel, and the choice of $H$ was arbitrary, so there is an $H'$ parallel to any $m$-dimensional affine space.
\end{proof}

\begin{definition}
A polytope $P$ is \textit{h-nearly simple} if and only if $\deg_{G_P}(v)=\dim(P)$ for all but exactly $h$ vertices of $P$.
\end{definition}

\begin{definition}
Let $P$ be a polytope with face $F$. \textit{P truncated at F} is the polytope obtained by intersecting $P$ with a halfspace which does not contain the vertices of $F$ and whose interior contains the vertices of $P$ that are not contained in $F$.
\end{definition}

The face lattice of $P$ truncated at $F$ is obtained from the face lattice of $P$ by the following sequence of operations:

Let $I$ be the half open interval $(\emptyset, F]$. For each face $K$ not in $I$ that is greater than a face in $I$, create a face $K'$. The covering relations for $K'$ are as follows: If $J$ is a face of the lattice that is greater than some face in $I$ and $J$ covers $K$, then $J'$ covers $K'$ and $K$ covers $K'$. Each 0-face created this way covers $\emptyset$. Then $I$ is deleted from the face lattice.

The geometric interpretation of this is that the face $F$ is removed and a new facet $F'$ is created. Each face of $P$ that was not a subface of $F$ but intersected $F$ non-trivially is preserved, and creates a face of dimension one smaller than itself where it intersects the new hyperplane.

\begin{lemma}
\label{trunc}
Let $P$ be a d-dimensional polytope with vertex $x$. Let $T$ be a face of $P$ containing $x$. Let $P'$ be $P$ truncated at $T$ and let $F$ be the facet of $P'$ that is entirely contained in the boundary of the halfspace added by truncation. Let $w$ be a vertex of $P$ adjacent to $x$ not contained in $T$. Let $v$ be the vertex of $P'$ contained in $F$ and adjacent to $w$ considered as a vertex in $P'$. Then $P'$ is simple at $v$ if $P$ is simple at $w$.
\end{lemma}

\begin{proof}
Assume $P$ is simple at $w$.The edges of $F$ are the intersections of the new hyperplane with 2-faces of $P$. The edges of $F$ containing $v$ are contained in 2-faces of $P$ containing the edge between $x$ and $w$. Since $P$ is simple at $w$, any subset of edges incident to it define a face with dimension equal to the number of edges in the subset. There are exactly $d-1$ pairs of edges incident to $w$ that have $xw$ as one of the pair. So there are exactly $d-1$ 2-faces of $P$ containing $xw$. So there are exactly $d-1$ edges in $F$ containing $v$.  The edge $vw$ is the unique edge of $P'$ not in $F$ containing $v$. So $v$ is contained in exactly $d$ edges of $P'$. So $P'$ is simple at $v$.
\end{proof}

\begin{corollary}
If $P$ is simple at all vertices adjacent to $x$, then $P$ truncated at $x$ is simple at all vertices in $F$.
\end{corollary}

\begin{definition}
A \textit{$k$-regular induced connected subgraph} ($k$-rics) initial with respect to an acyclic orientation $O$ of a graph $G$ is a subgraph $J$ of $G$ such that:
\begin{enumerate}
\item $V_J$ is an initial set of $G$ with respect to $O$.

\item $J$ is an induced subgraph of $G$.

\item $J$ is $k$-regular and connected.
\end{enumerate}
\end{definition}

\begin{theorem}
\label{1simple}
Let $P$ be a 1-nearly simple polytope. Then the face lattice of $P$ can be determined from the graph of $P$.
\end{theorem}

Let $x$ be the single vertex of $P$ at which $P$ is not simple. To prove this, we will use the graph to determine the 2-faces of $P$ containing $x$. Then the graph of $P$ truncated at $x$ will be determined. By Lemma \ref{trunc}, $P$ truncated at $x$ is a simple polytope. Then we use Theorem \ref{Kalai} to reconstruct the face lattice of $P$ truncated at $x$, and finally use the knowledge of which facet was created to obtain the face lattice of $P$.

\begin{proof}

Let $x$ be the single vertex of $P$ at which $P$ is not simple. Let $O$ be a linear functional orientation with respect to which $x$ is an initial set of $G_P$. Since $x$ is an initial set, indeg$(x) =0$. The polytope $P$ is simple at every vertex of $P$ except $x$. Since $O$ is a linear functional orientation, each face has a unique sink. Since $P$ is simple at all vertices except $x$, each face is simple at its sink. Furthermore, since $P$ is simple at all vertices except $x$, for each vertex $v$ of $P$, the edges oriented towards $v$ determine a face of $P$ with $v$ as its sink which contains those edges. By Lemma \ref{fOimplies}, this gives an orientation of $G_P$ for which $f^O = \sum_{i=0}^d f_i$. Therefore any acyclic orientation $O$ of $P$ which minimizes $f^O$ must have $f^O = \sum_{i=0}^d f_i$. 

For each 2-face of $P$ containing $x$, there is a linear functional orientation $O$ for which the graph of that face is a 2-rics initial with respect to $O$, and for which $x$ is an initial set with respect to $O$.

Let $J$ be a 2-rics of $G_P$ which contains $x$ and that is initial with respect to some acyclic orientation $O$ which minimizes $f^O$. Since $O$ is acyclic, $J$ has a sink $v$. By Lemma \ref{fOimplies}, there is a 2-face of $P$, $F'$, with $v$ as its sink. Since $J$ is an initial subgraph, and $F'$ has a unique sink which is contained in $V_J$, each vertex of $F'$ comes before a vertex of $J$ in the partial order given by $O$. Therefore, $V_{F'} \subset V_J$. Since both $J$ and $G_{F'}$ are 2-regular connected subgraphs, $J = G_{F'}$, so $J$ is the graph of a 2-face of $P$ containing $x$.

Therefore, the graphs of 2-faces of $P$ containing $x$ are exactly the 2-rics of $G_P$ which contain $x$ and that are initial with respect to some $f^O$-minimizing acyclic orientation.

Let $P'$ be $P$ truncated at $x$, with $F$ the resulting facet. Since $P$ is simple at every vertex except $x$, by Lemma \ref{trunc}, $P'$ is a simple polytope. The edges of $F$ can be determined from the 2-faces of $P$ containing $x$. Since the 2-faces of $P$ containing $x$ can be determined from the graph of $P$, the graph of $P'$ can be determined from the graph of $P$. Since $P'$ is simple, the face lattice of $P'$ can be determined from the graph of $P$ by Theorem \ref{Kalai}.

Since $P'$ is $P$ truncated at $x$, the entire face lattice of $P$, excluding $x$, appears in the face lattice of $P'$. To obtain the face lattice of $P$ from the face lattice of $P'$, in each face, replace the interval $(\emptyset, F]$ with $x$, with order relations for $x$ given as follows: If $G$ is a face of $P'$ and the intersection of $F$ and $G$ is not empty, then $G > x$. Furthermore $x > \emptyset$.

In conclusion, the face lattice of $P$ can be determined from the graph of $P$.
\end{proof}

\begin{definition}
A \textit{binding} is a $(d-2)$-dimensional affine space in a $d$-dimensional vector space.
\end{definition}

\begin{definition}
A \textit{page} is half of a $(d-1)$-dimensional affine space whose boundary is a binding in a $d$-dimensional vector space.
\end{definition}

\begin{definition}
A \textit{$y$-sink acyclic orientation} is an acyclic orientation for which $y$ as a sink of $G_P$.
\end{definition}

\begin{definition}
A \textit{$y$-sink $f^O$-minimizing acyclic orientation} is a $y$-sink acyclic orientation which minimizes $f^O$ among all $y$-sink acyclic orientations.
\end{definition}

\begin{lemma}
\label{ysink}
Let $P$ be a 2-nearly simple polytope with $y$ one of its non-simple vertices. Define $K := 2^{\deg(y)}-|\{\text{faces containing }y\}|$. Then there exist $y$-sink acyclic orientations $O$ with $f^O = \sum_{i=0}^d f_i(P) + K$.
\end{lemma}
\begin{proof}
Let $P$ be a $d$-dimensional 2-nearly simple polytope.
Let $x$ and $y$ be the two non-simple vertices of $P$.

Define $K := 2^{\deg(y)}-|\{\text{faces containing }y\}|$.

Let $F$ be a 2-face of $P$ containing $x$ and not containing $y$. Let $\mathcal{H}$ be the set of all hyperplanes that separate $F$ from the other vertices of $P$ and let $\mathcal{J}$ be the set of all hyperplanes that separate $y$ from the other vertices of $P$. Let $B$ be the set of bindings that are the intersection of some pair of elements of $\mathcal{H}$ and $\mathcal{J}$ and whose intersection with $P$ is the empty set. Let $H$ and $J$ be a pair of hyperplanes in $\mathcal{H}$ and $\mathcal{J}$, such that $H$ and $J$ are not parallel, each point in $H \bigcap P$ is within $\epsilon$ of a point in $F$ and each point in $J \bigcap P$ is within $\epsilon$ of $y$. Then $H \bigcap J$ is a binding not intersecting $P$, so $B$ is nonempty. Let $b= H \bigcap J$. Let $\vec{n}_H$ and $\vec{n}_J$ be normal vectors of $H$ and $J$. Let $\pi_b$ be the projection from $\mathbb{R}^d$ to a 2-dimensional vector space such that $\pi_b(b)$ is a single point. Since $H$ and $J$ are hyperplanes, $\vec{n}_H$ and $\vec{n}_J$ are unique up to a scalar. There exists some $\epsilon$ such that $H+\delta\vec{n}_H$ is in $\mathcal{H}$ for all $\delta$ such that $|\delta| < \epsilon$ and such that $J+\delta\vec{n}_J$ is in $\mathcal{J}$ for all $\delta$ such that $|\delta| < \epsilon$. The intersection of $H+\delta\vec{n}_H$ and $J+\delta\vec{n}_J$ is a binding in $B$ that is parallel with $b$.
Let $T=\{\pi_b(s)|s=\left(H+\delta_1\vec{n}_H\right) \bigcap\left( J+\delta_2\vec{n}_J \right), |\delta_1|,|\delta_2| < \epsilon\}$. Then $T$ is an open set that contains all points within $\epsilon$ of $\pi_b(b)$ in the image of $\pi_b$. Let $\mathfrak{K'}$ be the set of lines between vertices of $P$, and let $\mathfrak{K}=\{\pi_b(K) | K \in \mathfrak{K'}\}$. Then we can apply Lemma \ref{affine} to obtain a point $e$ in $T$ avoiding $\mathfrak{K}$. Let $E$ be the binding that projects to $e$ under $\pi_b$. Therefore $E$ is a binding that is the intersection of a hyperplane in $\mathcal{H}$ and a hyperplane in $\mathcal{J}$ that does not intersect any line between a pair of vertices of $P$.
From $E$, we will derive an orientation of $G_P$ by defining a total order on the vertices of $G_P$. Take $T$ to be the page containing $y$ with binding $E$. The order on vertices of $P$ is given by the following function:
$$ f: V_P \rightarrow [0,\pi)$$
$$ f(v) := \text{the minimum angle between }T\text{ and the page with binding }E\text{ that contains }v$$

By the construction of $E$, $f(v)$ is injective. Let $O$ be the orientation of $G_P$ that orients the edge $vw$ towards $v$ when $f(v) < f(w)$ and towards $w$ when $f(v) > f(w)$. We will call such orientations book orientations. Since $P$ is convex, and its vertices are totally ordered under $f(v)$, it has a unique sink with respect to $O$. This sink is $y$, since $f(y)=0$ and $f(v)$ is injective with domain $[0,\pi)$. The page contained in the hyperplane $H$ used to define $E$ has a larger angle with the page with binding $E$ containing $y$ than $f(v)$ for any vertex $v$ not in $F$, and a smaller angle than $f(w)$ for any vertex $w$ in $F$.  Under $O$, each face of $P$ has a unique sink. For every vertex of $P$ except $y$, $2^{\indeg(v)}=|\{$faces of $P$ with $v$ a sink of the face with respect to $O\}|$. For $y$, $2^{\indeg(y)}=|\{$faces of $P$ with $v$ a sink of the face with respect to $O\}|+K$. Therefore, summing over vertices gives $f^O = \sum_{i=0}^d f_i + K$.
\end{proof}

\begin{theorem}
\label{2simple}
Let $P$ be a 2-nearly simple polytope. Then the face lattice of $P$ can be determined from the graph of $P$.
\end{theorem}

To prove this, we first determine all the 2-faces of $P$ containing exactly one non-simple vertex. Then we will split the proof into two cases. In each case, we determine the 2-faces that might contain both non-simple vertices.  If the non-simple vertices are adjacent, we truncate the edge between the non-simple vertices, which will result in a simple polytope. Then Theorem \ref{Kalai} constructs the face lattice, and we undo the truncation to arrive at the original polytope's face lattice. If the non-simple vertices are not adjacent, we truncate one of the vertices, then use Theorem \ref{1simple} to reconstruct the face lattice and undo the truncation to arrive at the original polytope's face lattice.

\begin{proof}

Let $P$ be a $d$-dimensional 2-nearly simple polytope.
Let $x$ and $y$ be the two non-simple vertices of $P$.

Define $K := 2^{\deg(y)}-|\{\text{faces containing }y\}|$.

 For $O$ a $y$-sink $f^O$-minimizing acyclic orientation, $y$ is a sink of $G_P$ with respect to $O$. Therefore, the following chain of equalities hold: $2^{\indeg(y)} = 2^{\deg(y)} =$ 
$|\{\text{faces containing }y\}|+K=$  \linebreak $ |\{\text{faces with }y\text{ a sink of that face with respect to O}\}|+K$. Repeating the proof of Lemma \ref{fOfaces} with this change gives $f^O \geq \sum_{i=0}^d f_i(P) + K$.

By Lemma \ref{ysink}, this bound is obtained and $y$-sink $f^O$ minimizing acyclic orientations have $f^O = \sum_{i=0}^d f_i(P) + K$.

Using the construction given in Lemma \ref{ysink}, each 2-face of $P$ containing $x$ that does not contain $y$ is an initial set of $G_P$ with respect to some $y$-sink $f^O$-minimizing orientation. So each 2-face of $P$ containing $x$ that does not contain $y$ is a 2-rics initial with respect to a $y$-sink $f^O$-minimizing orientation.

Let $F$ be a 2-rics of $G_P$ containing $x$ and not containing $y$, with $V_F$ an initial set with respect to a $y$-sink $f^O$-minimizing acyclic orientation. Since $f^O = \sum_{i=0}^d f_i + K$, and $y$ is a sink of $O$, for every vertex $v$ of $P$ that is not $y$, every subset of edges oriented towards $v$ defines a face of $P$ for which $v$ is the unique sink. Because $F$ has a sink, and that sink is not $y$, the two edges in $F$ incident to that sink form a subset of the edges oriented towards the sink, and so define a 2-face of $P$ which also has that vertex as a sink of its graph. Label the graph of this face $F'$. Each face of $P$ has a unique sink, so each vertex of $F'$ has a path following the orientation to the sink. Since $V_F$ is an initial set, $F' \subset F$. Then, since $F$ and $F'$ are both 2-regular connected graphs, $F' = F$. So $F$ is the graph of a face of $P$.

This argument can be repeated switching the roles of $x$ and $y$.

This argument gives that the 2-faces of $P$ that contain exactly one of $x$ or $y$ are exactly the 2-rics that contain exactly one of $x$ or $y$ whose vertex sets are initial sets with respect to an $x$ or $y$-sink $f^O$-minimizing acyclic orientation. In other words, 2-faces of $P$ containing exactly one of $x$ or $y$ can be determined from $G_P$.

From here, we break the argument into two cases depending on whether $x$ and $y$ are adjacent in $G_P$.

Case 1: $x$ is not adjacent to $y$.
 
Since $x$ and $y$ are not adjacent, there is at most one 2-face of $P$ containing both $x$ and $y$. Every 2-face of $P$ containing $x$ and not $y$ can be determined from the graph of $P$. The graph of $P$ truncated at $x$ can be determined from the 2-faces containing $x$ and the graph of $P$. Since $y$ is not adjacent to $x$,  each vertex of $P$ adjacent to $x$ is simple. By Lemma \ref{trunc}, the resulting facet obtained by truncation is simple. Each vertex of the resulting facet has $d-1$ edges within the facet incident to it. Since there is at most one 2-face of $P$ containing $x$ and $y$, there is up to one edge of the new facet that cannot be determined by finding the 2-faces that contain $x$ but not $y$. Since the new facet is simple, the graph of the facet is $(d-1)$-regular, and all but one edge of the graph is known. Then there is a pair of vertices that are in only $d-2$ known edges. Those vertices must then be adjacent. Therefore we can determine the graph of $P$ truncated at $x$. By Theorem \ref{1simple}, we can determine the face lattice of this polytope, which along with the graph of $P$ determines the face lattice of $P$.

Case 2: $x$ and $y$ are adjacent in $G_P$. Therefore $xy$ is a 1-face of $P$.

Assume $x$ and $y$ are adjacent. Then there is a linear functional $g$ for which $g(x) < g(y) < g(v)$ for each vertex $v$ such that $P$ is simple at $v$. Then from this linear functional, we obtain a linear functional orientation $O$ of $G_P$.  Since $\indeg(x) = 0$ and $\{x\}$ is the only face of $P$ with $x$ as its sink, $2^{\indeg(x)}=|\{$faces $F$ of $P | x$ is a sink of $F$ with respect to $O\}|$. Likewise, since $\indeg(y) = 1$, $\{y\}$ and $\{xy\}$ are the only two faces of $P$ with $y$ as their sink, $2^{\indeg(y)}=|\{$faces $F$ of $P | y$ is a sink of $F$ with respect to $O\}|$. Since $P$ is simple at each of its other vertices, $2^{\indeg(v)}=|\{$faces $F$ of $P | v$ is a sink of $F$ with respect to $O\}|$ for $v \neq x,y$. Therefore, $f^O = \sum_{v \in V_P} |\{$faces $F$ of $P | v$ is a sink of $F$ with respect to $O\}|$. Each face of $P$ has a unique sink with respect to $O$, since $O$ is a linear functional orientation. Therefore $\sum_{v \in V_P} |\{$faces $F$ of $P | v$ is a sink of $F$ with respect to $O\}| = \sum_{i=0}^d f_i$. So this orientation has $f^O =  \sum_{i=0}^d f_i$. So the minimum of $f^O$ among all acyclic orientations of $G_P$ is  $\sum_{i=0}^d f_i$ by Lemma \ref{fOfaces}. 

Let $F'$ be a 2-face of $P$ containing $x$ and $y$. Then there is a linear functional orientation with respect to which $V_{F'}$ is an initial set. For such an orientation, $f^O=\sum_{i=0}^d f_i$. This is because $2^{\indeg(x)}= |\{\text{faces with }x\text{ a sink of that face with respect to O}\}|$, and $2^{\indeg(y)}=|\{$faces with $y$ a sink of that face with respect to $O\}|$, since they lie on an initial 2-face.  The linear functional orientation given by a vector rotated $\epsilon$ from normal to $F$ is one such orientation. 

Let $J$ be a 2-rics of $G_P$ that contains $x$ and $y$ and is an initial set with respect to an $f^O$-minimizing acyclic orientation. Let $v$ be a sink of $J$. Since $f^O = \sum_{i=0}^d f_i$, by Lemma \ref{fOimplies}, there is a 2-face $F'$ of $P$ with $v$ as its unique sink with respect to $O$. Since $v$ is the unique sink of $G_{F'}$, each vertex of $F'$ has a path to
$v$ along $O$. Since $V_J$ is an initial set of $G_P$, $V_{F'} \subset V_J$. Then, since $J$ and $G_{F'}$ are 2-regular induced connected graphs, $G_{F'} = J$. So $J$ is the graph of a face of $P$.

Therefore every 2-face of $P$ containing $x$ and not $y$, containing $y$ and not $x$, or containing $x$ and $y$ can be determined from $G_P$. This is sufficient information to determine the graph of $P$ truncated at $xy$.

Let $P'$ be $P$ truncated at $xy$. By Lemma \ref{trunc}, $P'$ is simple at each vertex of the new facet, since $P$ is simple at each vertex not $x$ or $y$. So $P'$ is a simple polytope. Since the graph of $P'$ can be determined from the graph of $P$, the face lattice of $P'$ can be determined from the graph of $P$. The face lattice of $P$ can be determined from the face lattice of $P$ truncated at $xy$ and the graph of $P$. So the face lattice of $P$ can be determined from the graph of $P$.
\end{proof}

\section{Computational results}

\begin{example}
\label{example}

The following pair of 4-polytopes are 3-nearly simple combinatorially distinct polytopes with isomorphic graphs.
Furthermore, this is the pair of polytopes with the smallest number of faces that are combinatorially distinct and have isomorphic graphs. This is proved by exhaustion of all smaller combinatirially distinct polytopes listed in full by Miyata et al. \cite{Miyata}.

$P_3$ has 8 vertices labeled 0 through 7. Below its facets are given by the vertices they contain:
\begin{itemize}
\item $F_{1,1} = [7, 6, 5, 4, 3, 2]$
\item $F_{1,2} = [7, 6, 5, 4, 1, 0]$
\item $F_{1,3} = [7, 6, 3, 2, 1]$
\item $F_{1,4} = [7, 5, 3, 1, 0]$
\item $F_{1,5} = [6, 4, 2, 1]$
\item $F_{1,6} = [5, 4, 3, 0]$
\item $F_{1,7} = [4, 3, 2, 1, 0]$
\end{itemize}

$P_4$ has 8 vertices labeled 0 through 7. Below its facets are given by the vertices they contain:
\begin{itemize}
\item $F_{2,1} = [7, 6, 5, 4, 3, 2]$
\item $F_{2,2} = [7, 6, 5, 4, 1, 0]$
\item $F_{2,3} = [7, 6, 3, 2, 1]$
\item $F_{2,4} = [7, 5, 3, 1, 0]$
\item $F_{2,5} = [6, 4, 2, 1]$
\item $F_{2,6} = [5, 4, 3, 0]$
\item $F_{2,7} = [4, 3, 2, 1]$
\item $F_{2,8} = [4, 3, 1, 0]$
\end{itemize}

This pair of polytopes are identical except for the pair of structures involving $F_{1,7}$ and $F_{2,7}$ and $F_{2,8}$. In $P_3$, $F_{1,7}$ is the bipyramid over a triangle, whereas in $P_4$, $F_{2,7}$ and $F_{2,8}$ are simplices, corresponding to the upper and lower pyramid over the triangle. In general, this pair of structures gives identical graphs of polytopes and can be the bipyramid over any polygon or two pyramids over the same polygon.
 
\end{example}

With the help of Sage and the data provided by Miyata, Moriyama and Fukuda \cite{Miyata}, we can explicitly state polytopes that are uniquely determined by their graph. Since there are a finite number of polytopes on a fixed number of vertices, computing the graphs of each polytope with $k$ vertices gives a map from polytopes on $k$ vertices to graphs on $k$ vertices. When there is only one polytope mapped to a specific graph, we say that the polytope is reconstructible from that graph. This method gives no way to construct such a polytope from its graph without already knowing all polytopes with that many vertices. This is rather unhelpful in general, since if we knew all polytopes on a certain number of vertices, there would be little practical purpose to construct these polytopes from their graphs. However, this may give rise to a future insight into what is required to reconstruct a polytope from its graph. Below is the complete list of 4-polytopes on 7 vertices reconstructible from their graph, given as lists of facets with the vertices contained in that facet.

$P_5$: [[6, 5, 4, 3, 2, 1], [6, 5, 4, 3, 0], [6, 5, 2, 0], [6, 4, 2, 0], [5, 3, 1, 0], [5, 2, 1, 0], [4, 3, 1, 0], [4, 2, 1, 0]]

$P_6$: [[6, 5, 4, 3, 2, 1], [6, 5, 4, 3, 0], [6, 5, 2, 1, 0], [4, 3, 2, 1, 0], [6, 4, 2, 0], [5, 3, 1, 0]]

$P_7$: [[6, 5, 4, 3, 2, 1], [6, 5, 4, 3, 0], [6, 5, 2, 1, 0], [6, 4, 2, 0], [5, 3, 1, 0], [4, 3, 2, 0], [3, 2, 1, 0]]

$P_8$: [[6, 5, 4, 3, 2, 1], [6, 5, 4, 3, 2, 0], [6, 5, 1, 0], [6, 4, 1, 0], [5, 3, 1, 0], [4, 2, 1, 0], [3, 2, 1, 0]]

$P_9$: [[6, 5, 4, 3, 2], [6, 5, 4, 3, 1], [6, 5, 2, 1, 0], [6, 4, 2, 1, 0], [5, 3, 2, 0], [5, 3, 1, 0], [4, 3, 2, 0], [4, 3, 1, 0]]

$P_8$ is 2-nearly simple, $P_6$ is 1-nearly simple, and $P_7, P_5, P_9$ are 3, 5, and 6-nearly simple respectively. 

We do not list the 34 3-polytopes on 7 vertices that are reconstructible from their graphs here.

There are three further polytopes on 7 vertices that are reconstructible from their graph and dimension. They are:

$P_{10}$: [[6, 5, 4, 3], [6, 5, 4, 2], [6, 5, 3, 2], [6, 4, 3, 1], [6, 4, 2, 1], [6, 3, 2, 0], [6, 3, 1, 0], [6, 2, 1, 0], [5, 4, 3, 1], [5, 4, 2, 0], [5, 4, 1, 0], [5, 3, 2, 0], [5, 3, 1, 0], [4, 2, 1, 0]]

$P_{11}$: [[6, 5, 4, 3, 2, 1], [6, 5, 4, 3, 2, 0], [6, 5, 4, 3, 1, 0], [6, 5, 2, 1, 0], [6, 4, 2, 1, 0], [5, 3, 2, 1, 0], [4, 3, 2, 1, 0]]

$P_{12}$: [[6,5,4,3,2,1], [6,5,4,3,2,0], [6,5,4,3,1,0], [6,5,4,2,1,0], [6,5,3,2,1,0], [6,4,3,2,1,0], [5,4,3,2,1,0]] 

$P_{10}$ is a 4-polytope, $P_{11}$ is a 5-polytope, and $P_{12}$ is the 6-simplex. Each is the only polytope of their dimension mapped to their graph, but there are polytopes of a different dimension that are also mapped to their graph. This is more similar to the question of reconstructing simple polytopes, because the information that the polytope is simple along with its graph is equivalent to the information of its dimension along with its graph. 

\section{Conclusion}

With the definition of $h$-nearly simple convex polytopes, a broader description of which polytopes are reconstructible from their graph and which are not may be given. A 0,1, or 2-nearly simple convex polytope can be reconstructed from its graph. There exist 3-nearly simple convex polytopes which can not be reconstructed from their graphs. These results can be expanded upon in several ways. The essential next question is ``Is there a complete characterization of convex polytopes that can be reconstructed from their graph?" Some smaller questions are, ``Can these results be extended to the $k$-skeletons of polytopes?" and ``Are there any 3-nearly simple polytopes that cannot be reconstructed from their graphs whose non-simple vertices do not induce $K_3$?"

\section{Acknowledgements}

Many thanks to Margaret Bayer for pushing me to strive for more and keeping my thinking straight. 

\bibliographystyle{plain}
\bibliography{paper.bib}

\end{document}